\documentclass{amsart}
\usepackage{amssymb}
\usepackage{amscd}
\usepackage[latin1]{inputenc}
\usepackage{amsmath}
\usepackage{amsfonts}
\usepackage{latexsym}
\usepackage{verbatim}
\usepackage{graphicx}
\usepackage{epsfig}
\usepackage{faktor}
\usepackage{xcolor}
\usepackage{soul}

\colorlet{RED}{red}
\colorlet{BLUE}{blue}
\setstcolor{red}

\newtheorem{theorem}{Theorem}[section]
\newtheorem{corollary}[theorem]{Corollary}
\newtheorem{lemma}[theorem]{Lemma}

\theoremstyle{definition}
\newtheorem{definition}[theorem]{Definition}

\newtheorem{remark}[theorem]{Remark}

\newcommand{\zz}{\mathbb{Z}}
\newcommand{\grp}{\faktor{\textrm{Aut}(X,\sigma)}{\langle \sigma\rangle}}
\newcommand{\aut}{\textrm{Aut}(X,\sigma)}

\begin{document}

\title[Local finiteness and aut. groups of low complexity subshifts]{Local finiteness and automorphism groups of low complexity subshifts}

\author{Ronnie Pavlov}
\address{Ronnie Pavlov\\
Department of Mathematics\\
University of Denver\\
2390 S. York St.\\
Denver, CO 80208}
\email{rpavlov@du.edu}
\urladdr{http://www.math.du.edu/$\sim$rpavlov/}

\author{Scott Schmieding}
\address{Scott Schmieding\\
Department of Mathematics\\
University of Denver\\
2390 S. York St.\\
Denver, CO 80208}
\email{scott.schmieding@du.edu}

\thanks{The first author gratefully acknowledges the support of a Simons Foundation Collaboration Grant.}
\renewcommand{\subjclassname}{MSC 2020}

\begin{abstract}
We prove that for any transitive subshift $X$ with word complexity function $c_n(X)$,
if $\liminf \frac{\log (c_n(X)/n)}{\log \log \log n} = 0$, then the quotient group $\grp$ of the automorphism group of $X$ by the subgroup generated by the shift $\sigma$ is locally finite. We prove that significantly weaker upper bounds on $c_n(X)$ imply the same conclusion if the Gap Conjecture from geometric group theory is true.

Our proofs rely on a general upper bound for the number of automorphisms of $X$ of range $n$ in terms of word complexity, which may be of independent interest. As an application, we are also able to prove that for any subshift $X$, if
$\frac{c_n(X)}{n^2 (\log n)^{-1}} \rightarrow 0$, then $\aut$ is amenable, improving a result of Cyr and Kra.

In the opposite direction, we show that for any countable infinite locally finite group $G$ and any unbounded increasing $f: \mathbb{N} \rightarrow \mathbb{N}$, there exists a minimal subshift $X$ with $\grp$ isomorphic to $G$ and $\frac{c_n(X)}{nf(n)} \rightarrow 0$.

\end{abstract}

\maketitle


\section{Introduction}

This work deals with symbolic dynamics, which is the study of symbolically defined topological dynamical systems called \textbf{subshifts}. A subshift is simply a closed and shift-invariant subset of $\mathcal{A}^{\mathbb{Z}}$ for some finite set $\mathcal{A}$. One way of measuring the size of a subshift $X$ is via its word complexity function $c_n(X)$; $c_n(X)$ is the number of different $n$-letter strings (or \textbf{words}) appearing within points of $X$.

Another sense of `complexity' for a subshift comes from its group of automorphisms; an \textbf{automorphism} of a subshift $X$ is a homeomorphism from $X$ to itself which commutes with the shift map $\sigma: X \rightarrow X$ defined by $(\sigma x)(n) = x(n+1)$. The set of automorphisms $\aut$ has an obvious group structure from composition, and turns out to always be countable. By definition, $\sigma$ itself is always in $\aut$, and $\langle \sigma \rangle$, the subgroup generated by $\sigma$, is always a normal subgroup of $\aut$. (See Section~\ref{defs} for more details.)

In this paper, we continue a line of research which has been fruitfully developed in many recent works (\cite{CKlinear}, \cite{CKstretch}, \cite{CKquad}, \cite{CKslow}, \cite{DDMP1}, \cite{DDMP2}, \cite{saloFG}, \cite{salofull}), namely:
in what sense must subshifts with low complexity functions have automorphism groups which are `small' or restricted? Though we do not claim it to be complete, we summarize some recent results in this area. Some of the following results include the hypothesis of transitivity/minimality of $X$; we postpone definitions to Section~\ref{symb}.

\begin{enumerate}
\item $X$ minimal and $\liminf \frac{c_n(X)}{n} < \infty$ $\Longrightarrow$ $\grp$ finite (\cite{CKlinear}, \cite{DDMP1})
\item $X$ transitive and $\limsup \frac{c_n(X)}{n} < \infty$ $\Longrightarrow$ $\grp$ finite (\cite{CKlinear})
\item $\limsup \frac{c_n(X)}{n} < \infty$ $\Longrightarrow$ all f.g. subgroups of $\aut$ are virtually $\mathbb{Z}^d$ (\cite{CKlinear})
\item $\frac{c_n(X)}{n^2 (\log n)^{-2}} \rightarrow 0$ $\Longrightarrow$ $\aut$ amenable (\cite{CKslow})
\item $X$ transitive and $\liminf \frac{c_n(X)}{n^2} = 0$ $\Longrightarrow$ $\grp$ periodic (\cite{CKquad})
\item $\liminf \frac{c_n(X)}{n^2} = 0$ $\Longrightarrow$ $\aut$ does not contain a free semigroup on two generators (\cite{CKslow})
\item $X$ minimal and $\frac{c_n(X)}{n^3} \rightarrow 0$ $\Longrightarrow$ every f.g. torsion-free subgroup of $\aut$ is virtually abelian (\cite{CKstretch})
\item $X$ minimal and $\exists d \in \mathbb{N}$ with $\frac{c_n(X)}{n^d} \rightarrow 0$ $\Longrightarrow$ $\aut$ amenable and every f.g. torsion-free subgroup of $\aut$ is virtually nilpotent (\cite{CKstretch})
\item $X$ minimal and $\exists \beta < 1/2$ with $\frac{\log c_n(X)}{n^{\beta}} \rightarrow 0$ $\Longrightarrow$ $\aut$ amenable (\cite{CKstretch})
\end{enumerate}

Here, we wish to add some more context to the transition from linear to slightly greater complexity function; to our knowledge, up to now, there have been no complexity thresholds used between the linear one for (1)-(3) above and $o(n^2/\log^2 n)$ from (4). It is reasonable to expect that complexity extremely close to linear should place restrictions on the group structure of $\grp$ which are stronger than periodicity. Our first main result shows that for transitive subshifts, low enough complexity in fact implies that $\grp$ is locally finite.

\begin{theorem}\label{mainthm2}
If $X$ is an infinite transitive subshift with $\liminf \frac{\log (c_n(X)/n)}{\log \log \log n} = 0$, then $\grp$ is locally finite
(and countable).
\end{theorem}

We briefly remark that local finiteness is a strictly stronger property than periodicity (for instance, the Tarski monster groups and Grigorchuk group are periodic but not locally finite), and so this result has a strictly stronger complexity hypothesis and conclusion than (5) above.

To prove Theorem \ref{mainthm2}, we first achieve some estimates on growth of number of automorphisms as a function of range by using left- and right-special words (Corollary~\ref{autbd}). We then use a theorem of Shalom and Tao (Theorem~\ref{ST}) to show that our growth is so slow as to force finitely generated subgroups of $\grp$ to be virtually nilpotent. Finally, we combine this with the fact that $\grp$ is known to be periodic under the hypotheses of Theorem~\ref{mainthm2} due to (5) above; thus all finitely generated subgroups of $\grp$ are virtually nilpotent and periodic, therefore finite.

There is a well-known conjecture in geometric group theory called the Gap Conjecture (see \cite{grigorgap}), which states that every finitely generated group with growth rate $e^{o(\sqrt{n})}$ (see Section~\ref{group2} for more details) has polynomial growth. The Gap Conjecture is known to hold for some classes of groups (\cite{grigorp}, \cite{grigorgap}, \cite{wilson1}, \cite{wilson2}), but is still open in general.

Variants of our first main result show that $\grp$ is locally finite under much weaker hypotheses if the Gap Conjecture is true.

\begin{theorem}\label{mainthm2b}
If $X$ is a transitive subshift with $\liminf \frac{c_n(X)}{n^{1.25} (\log n)^{-0.5}} = 0$ and the Gap Conjecture is true, then $\grp$ is locally finite (and countable).
\end{theorem}

\begin{theorem}\label{mainthm2c}
If $X$ is a transitive subshift with $\frac{c_n(X)}{n^{1.5} (\log n)^{-1}} \rightarrow 0$ and the Gap Conjecture is true, then $\grp$ is locally finite (and countable).
\end{theorem}


Finally, our techniques allow for a slight improvement to the theorem of Cyr and Kra referenced as (4) above, where they proved that if $(X,\sigma)$ is any subshift satisfying $c_n(X) = o(n^2/(\log^2 n))$, then $\aut$ is amenable. 


\begin{theorem}\label{mainthm2d}
If $X$ is a subshift with $\frac{c_n(X)}{n^2 (\log n)^{-1}} \rightarrow 0$, then $\aut$ is amenable (and countable).
\end{theorem}

Our final result is in the opposite direction, showing that no superlinear complexity threshold can impose stronger restrictions on $\grp$ than being locally finite (and countable).

\begin{theorem}\label{mainthm}
For any countable locally finite group $G$ and any unbounded increasing $f: \mathbb{N} \rightarrow \mathbb{R}$, there exists a minimal subshift $X$ with $\grp = G$ and $\frac{c_n(X)}{nf(n)} \rightarrow 0$.
\end{theorem}

In particular, Theorem \ref{mainthm} provides examples of minimal subshifts having arbitrarily slow but superlinear complexity function whose automorphism group is not virtually abelian, demonstrating that the words `finitely generated torsion-free' cannot be omitted in (7) above. For example, if one applies Theorem \ref{mainthm} in the case where $G$ is a countably infinite locally finite simple group, then in this case $\aut$ can not be virtually abelian.

\begin{remark}
Theorems~\ref{mainthm2} and \ref{mainthm} together completely characterize the possible $\grp$ for transitive subshifts $X$ with growth $n (\log \log n)^{o(1)}$ along a subsequence: they are exactly the locally finite groups.
\end{remark}

\begin{remark}
We would like to mention \cite{BD}, where they prove several results similar in spirit to Theorem~\ref{mainthm}, one of which realizes arbitrary Choquet simplices of invariant measures for (minimal) Toeplitz subshifts of arbitrarily low superlinear complexity. In addition to providing a class of examples satisfying our complexity assumptions in Theorems~\ref{mainthm2}-\ref{mainthm2d}, this also shows that there are subshifts with arbitrary (for instance very large) Choquet simplices and $\grp$ locally finite.
\end{remark}


\section{Definitions/preliminaries}\label{defs}

\subsection{Symbolic dynamics}\label{symb}


\begin{definition}
For any finite alphabet $\mathcal{A}$, the \textbf{full shift} over $\mathcal{A}$ is the set $\mathcal{A}^{\zz}$, which is viewed as a compact topological space with the (discrete) product topology.
\end{definition}

\begin{definition}
A \textbf{word} over $\mathcal{A}$ is a member $w \in \mathcal{A}^n$ for some $n \in \mathbb{N}$, which is referred to as its \textbf{length} and denoted by $|w|$. We say that a word $v$ is a \textbf{subword} of a word or biinfinite sequence $x$ if there exists $i$ so that $x([i, i+|w|)) = w$. (Here and throughout, all intervals are assumed to be intersected with $\mathbb{Z}$, e.g. $[2,5)$ represents $\{2,3,4\}$. For such an interval $I$, we view an element of $\mathcal{A}^I$ as a word of length $|I|$ by the obvious identification.)
\end{definition}

The set of words has an obvious binary operation of concatenation, and whenever we write expressions like $vw$ or $w^3$ it is with respect to concatenation.

\begin{definition}
The \textbf{left shift}, denoted by $\sigma$, is the self-map of the full shift defined by $(\sigma x)(n) = x(n+1)$ for
$x \in \mathcal{A}^{\mathbb{Z}}$ and $n \in \mathbb{Z}$.
\end{definition}

\begin{definition}
A \textbf{subshift} over $\mathcal{A}$ is a topological dynamical system $(X, \sigma)$ where $X$ is a closed subset of the full shift $\mathcal{A}^{\zz^d}$ (endowed with the subspace (product) topology) which is invariant under $\sigma$.
\end{definition}

Since there is never ambiguity about the dynamics on $X$, in this work we refer to a subshift simply by the space $X$ for ease of notation.

\begin{definition}
A word, one-sided infinite sequence, or bi-infinite sequence $x$ over $\mathcal{A}$ is \textbf{periodic with period $p$} if
$x(n) = x(n+p)$ for all $n \in \mathbb{Z}$ where both $x(n)$ and $x(n+p)$ are defined.
\end{definition}



\begin{definition}
The \textbf{language} of a subshift $X$, denoted by $L(X)$, is the set of all subwords of sequences in $X$. For all $n$, we write
$L_n(X) = L(X) \cap \mathcal{A}^n$ for the set of words of length $n$ in $L(X)$.
\end{definition}

\begin{definition}
A word $w$ is \textbf{right-special} for a subshift $X$ if there exist $a \neq b \in \mathcal{A}$ for which $wa, wb \in L(X)$. Similarly, $w$ is \textbf{left-special} for $X$ if there exist $c \neq d \in \mathcal{A}$ for which $cw, dw \in L(X)$. The sets of $n$-letter right-special and left-special words for $X$ are denoted by $RS_n(X)$ and $LS_n(X)$ respectively.
\end{definition}

\begin{definition}
The \textbf{word complexity sequence} of a subshift $X$ is defined by $c_n(X) := |L_n(X)|$.
\end{definition}

The following lemma is routine, but we include a proof for completeness. 

\begin{lemma}\label{compdiff}
For any subshift $X$ and $n \in \mathbb{N}$, $|LS_n(X)|$ and $|RS_n(X)|$ are less than or equal to $c_{n+1}(X) - c_n(X)$.
\end{lemma}

\begin{proof}
We give only the proof for $|RS_n(X)|$, as the other is trivially similar. Fix any $X$ and $n$, and consider the map
$f: L_{n+1}(X) \rightarrow L_n(X)$ removing the final letter of an $n+1$-letter word. This map is surjective, and it's clear from definition that $w \in L_n(X)$ is right-special iff its $f$-preimage has cardinality greater than $1$. From this it's immediate that
$|RS_n(X)| \leq c_{n+1}(X) - c_n(X)$.
\end{proof}

\begin{definition}
A subshift $X$ is \textbf{minimal} if for all $w \in L(X)$ and $x \in X$, $w$ is a subword of $x$.
\end{definition}

\begin{definition}
A subshift $X$ is \textbf{(topologically) transitive} if there exists $x \in X$ so that $X = \overline{\{\sigma^n x\}_{n \in \mathbb{Z}}}$.
\end{definition}

We briefly note that an infinite transitive subshift $X$ cannot have isolated periodic points; if $X$ is transitive, then there exists $x \in X$ for which $X = \overline{\{\sigma^n x\}}$. If $p \in X$ were isolated and periodic, then $p \in \{\sigma^n x\}$, implying that $X = \{\sigma^n p\}$ and that $X$ is finite by periodicity of $p$, a contradiction.

\begin{definition}
A {\bf (topological) factor map} from one subshift $X$ to another subshift $X'$ is a surjective continuous function
$\phi : X \rightarrow X'$ which commutes with the shift action (i.e. $\phi \circ \sigma = \sigma \circ \phi$).
\end{definition}

By the classical Curtis-Hedlund-Lyndon theorem, factor maps on subshifts have a very specific form.

\begin{theorem}\label{CHL}
For any factor map $\phi: X \rightarrow X'$, there exists $N$ and $\Phi: \mathcal{A}_X^{2N+1} \rightarrow \mathcal{A}_{X'}$ so that for all $x \in X$ and $n \in \mathbb{N}$, $(\phi x)(n) = \Phi(x([n-N, n+N]))$.
\end{theorem}

\begin{definition}
We say that a factor map $\phi$ has \textrm{range $N$} and \textrm{has inducing block map $\Phi$} if it satisfies the conclusion of Theorem~\ref{CHL}.
\end{definition}

We remark that though every factor $\phi$ has some range $N$ and inducing block map $\Phi$, these need not be unique.



\begin{definition}
An {\bf automorphism} of a subshift $X$ is a factor map from $X$ to itself which is bijective.
\end{definition}

If $\phi$ is an automorphism of $X$ with inducing block map $\Phi: \mathcal{A}_X^{2N+1} \rightarrow \mathcal{A}$, then for every word $w \in L_n(X)$ with $n \geq 2N+1$, we can let $\Phi$ act on $w$ as in the definition of $\phi$. Formally, let $\Phi(w)$ be the word of length $n - 2N$ defined by $(\Phi(w))(i) = \Phi(w([i-N, i+N]))$ for $N < i \leq n - N$.

We remark that ranges of automorphisms are additive under composition. Indeed, by definition, if $\phi$ has range $N$ and inducing map $\Phi$ and $\phi'$ has range $N'$ and inducing map $\Phi'$, then $\phi \circ \phi'$ has range $N+N'$ and inducing map $\Phi \circ \Phi'$ (where $\Phi'$ acts on words in $\mathcal{A}_X^{2N+2N'+1}$ as defined above.)



\subsection{Group theory}\label{group1}

We here summarize some basic definitions from group theory. We will not have need of any advanced group theory in this paper, so we do not go into great detail. For more information, see \cite{RobinsonGT}.

While we often make it explicit in the text, throughout we will assume groups to be countable and discrete. 

Clearly, for any subshift $X$, the set of automorphisms of $X$ form a group under the operation of composition, and we denote this group by $\aut$. Since $\sigma$ is itself in $\aut$ and all automorphisms commute with $\sigma$ by definition, the subgroup of $\aut$ generated by $\sigma$ is always normal in $\aut$, and so we may refer to $\grp$. We refer to the set (generally not a group) of automorphisms with range $n$ by $\textrm{Aut}_n(X, \sigma)$.

\begin{definition}
For a subset $S$ of a group $G$, we denote by $\langle S \rangle$ the subgroup of $G$ generated by $S$. A group $G$ is said to be \textbf{generated by} $S \subset G$ if $\langle S \rangle = G$. A group $G$ is \textbf{finitely generated} if there exists a finite $S \subset G$ for which $G = \langle S \rangle$.
\end{definition}

\begin{definition}
A group $G$ is called \textbf{locally finite} if every finitely generated subgroup of $G$ is finite.
\end{definition}

Any countable and locally finite group may be written as a countable increasing union of finite subgroups.

\begin{definition}
A group $G$ is called \textbf{periodic} if every element in $G$ has finite order.
\end{definition}


\begin{definition}
A countable 
group $G$ is \textbf{amenable} if there exists a sequence $F_{i} \subset G$ of finite subsets of $G$ such that, for every $g \in G$, 
\[
\lim_{i \to \infty}\frac{|F_{i} \Delta gF_{i}|}{|F_{i}|}=0.
\]
\end{definition}

\begin{definition}
A group $G$ is \textbf{nilpotent} if there exists a sequence of subgroups
\[
\{\textrm{id}\} = H_{0} \subset \cdots \subset H_{k-1} \subset H_{k} = G
\]
such that each $H_{i}$ is normal in $G$ and $H_{i+1}/H_{i}$ is contained in the center of $G/H_{i}$ for all $i$. A group $G$ is \textbf{virtually nilpotent} if it contains a finite index nilpotent subgroup.
\end{definition}

As noted in the introduction, there has been significant recent work on restrictions on $\aut$ imposed by the word complexity function of $X$. We mention one such result here which we will need in our proofs.

\begin{theorem}[\cite{CKquad}]\label{CK}
If $X$ is a transitive subshift and $\frac{c_n(X)}{n^2} \rightarrow 0$, then $\grp$ is a periodic group.
\end{theorem}


\begin{remark}
Theorem \ref{CK} is not necessarily true if one drops the transitivity assumption on the subshift, even when the complexity function grows linearly (here by a complexity function growing linearly we mean it is bounded above by some linear function). For example, if $(X_{1},\sigma_{1})$ and $(X_{2},\sigma_{2})$ are two disjoint infinite subshifts whose complexity functions grow linearly and $(Y,\sigma_{Y})$ is the union of $(X_{1},\sigma_{1})$ and $(X_{2},\sigma_{2})$, then the complexity function for $(Y,\sigma_{Y})$ also grows linearly, but $\faktor{\textrm{Aut}(Y,\sigma_{Y})}{\langle \sigma_{Y} \rangle}$ is not a periodic group, since the image of the automorphism $\sigma_{1} \times \textrm{id}$ under the quotient
 map $\textrm{Aut}(Y, \sigma_{Y}) \to \faktor{\textrm{Aut}(Y,\sigma_{Y})}{\langle \sigma_{Y} \rangle}$ is of infinite order.
\end{remark}

\subsection{Geometric group theory}\label{group2}

We summarize here some basic results from geometric group theory that we will need. For a more detailed introduction to this area, see~\cite{LohGGT}.

\begin{definition}
For any group $G$ generated by a finite set $S$ and any $n \in \mathbb{N}$, $B_n(S)$ denotes the set of `words of length at most $n$ over $S$', i.e.
\[
B_{n}(S) = \{g \in G \mid g = g_{1} \cdots g_{k} \textnormal{ for some } k \leq n \textnormal{ and } g_{i} \in S\}.
\]
\end{definition}

\begin{definition}
A finitely generated group $G$ has \textbf{polynomial growth} if there exists a finite generating set $S$ and constants $C$ and $d$ so that
$|B_n(S)| < Cn^d$ for all $n$.
\end{definition}

It is well-known that all virtually nilpotent groups have polynomial growth. A celebrated theorem of Gromov shows that the converse is also true.

\begin{theorem}[\cite{gromov}]\label{gromov}
If $G$ is a finitely generated group with polynomial growth, then $G$ is virtually nilpotent.
\end{theorem}

The following theorem of Shalom and Tao shows that there is an explicit superpolynomial rate below which growth rates must be polynomial.

\begin{theorem}[\cite{ST}, Corollary 1.10]\label{ST}
There exists a constant $c > 0$ so that if $G$ is a group generated by a finite subset $S$, and there exists $N > c^{-1}$ for which $|B_{N}(S)| \leq N^{c (\log \log N)^c}$, then $G$ is virtually nilpotent.
\end{theorem}

Although Theorem~\ref{ST} is the first result that gives an explicit `gap' in growth rates for finitely generated groups (i.e. there is no finitely generated group with growth greater than polynomial but lower than $N^{c (\log \log n)^c}$), it is conjectured that this gap is much larger. The Gap Conjecture (\cite{grigorgap}) states that if a group has finite generating set $S$ and
$|B_n(S)| = e^{o(\sqrt{n})}$, then in fact $G$ has polynomial growth (and is therefore virtually nilpotent by Gromov's Theorem). The Gap Conjecture is still open, but it is known to hold for some classes of groups (\cite{grigorp}, \cite{grigorgap}, \cite{wilson1}, \cite{wilson2}).


\begin{definition}
A finitely generated group $G$ has \textbf{subexponential growth} if there exists a generating set $S$ so that
$\frac{\log |B_n(S)|}{n} \rightarrow 0$.
\end{definition}

It is well-known that finitely generated groups of subexponential growth must be amenable (for instance, see~\cite[Cor. 9.2.4]{LohGGT}).

\section{$\aut$ in the low complexity setting}
In this section, we prove Theorems~\ref{mainthm2}-\ref{mainthm2d}. The main tool for Theorems~\ref{mainthm2}-\ref{mainthm2c} is Theorem~\ref{ST}, combined with the following lemma, which bounds the number of automorphisms of a given range.

\begin{corollary}\label{autbd}
For every infinite transitive subshift $X$ and every $n$,
\[
|\textrm{Aut}_{\lfloor (n-1)/2 \rfloor}(X, \sigma)| \leq (c_{1+c_n(X)}(X))^{2|A| (c_{n+1}(X) - c_n(X))}.
\]
\end{corollary}

This is actually a corollary of the following slightly more general theorem, which we will need for Theorem~\ref{mainthm2d}.

For any subshift $X$, define $\textrm{Aut}^{(FIP)}(X, \sigma) \subset \aut$ to be the subgroup of automorphisms of $X$ which fix all isolated periodic points in $X$. (If $X$ has no isolated periodic points, then $\textrm{Aut}^{(FIP)}(X, \sigma) := \aut$.)
We denote the set of such automorphisms which have range $n$ by $\textrm{Aut}^{(FIP)}_n(X, \sigma)$.
\begin{theorem}\label{autbd2}
Let $X$ be a subshift. Then for every $n$,
\[
|\textrm{Aut}^{(FIP)}_{\lfloor (n-1)/2 \rfloor}(X, \sigma)| \leq (c_{1+c_n(X)}(X))^{2|A| (c_{n+1}(X) - c_n(X))}.
\]
\end{theorem}

\begin{proof}

For any subshift $X$ and any $n$, define an \textbf{$n$-right branch word} to be a word in $L(X)$ beginning with a word in $RS_n(X)$, containing no other word in $RS_n(X)$, containing no repeated $n$-letter subwords, and which is maximal with respect to subword inclusion subject to these constraints. Similarly, define an \textbf{$n$-left branch word} to be a word in $L(X)$ ending with a word in $LS_n(X)$, containing no other word in
$LS_n(X)$, with no repeated $n$-letter subwords, and which is maximal with respect to subword inclusion subject to these constraints. An $n$-branch word is any word that is either an $n$-left or $n$-right branch word.

The proof relies on the following three facts about $n$-branch words.

\begin{enumerate}
\item For every $n$, the number of $n$-branch words is less than or equal to the quantity $2|A| (c_{n+1}(X) - c_n(X))$. 

\item For every $n$, each $n$-branch word has length less than $n + c_n(X)$. 

\item Suppose $\phi_1$ and $\phi_2$ are automorphisms with range $\lfloor (n-1)/2 \rfloor$ induced by block codes $\Phi_{1}, \Phi_{2}$ respectively such that $\phi_{1}, \phi_{2}$ fix all isolated periodic points, and $\Phi_1(w) = \Phi_2(w)$ for all $n$-branch words $w$. Then $\phi_1 = \phi_2$. 
\end{enumerate}

\noindent
\textbf{Proof of (1):} Each $n$-right branch word $w$ is determined completely by its initial word in $RS_n(X)$ and the following letter; then, since $w$ contains no other words in $RS_n(X)$, each $n$-letter subword determines the next letter, meaning that all of $w$ is forced. There are obviously at most $|A| |RS_n(X)|$ choices for this initial word and following letter, which is less than or equal to $|A| (c_{n+1}(X) - c_n(X))$ by Lemma~\ref{compdiff}. A similar bound holds for $n$-left branch words, implying (1).\\

\noindent
\textbf{Proof of (2):} Every $n$-branch word contains no repeated $n$-letter subwords, and so contains at most $c_n(X)$ $n$-letter subwords. This clearly implies that such a word has length less than $n + c_n(X)$.\\

\noindent
\textbf{Proof of (3):} We claim that every $w \in L_n(X)$ which is not the subword of any $n$-branch word must be a subword of an isolated periodic point of $X$. To see this, assume that $w \in L_n(X)$ is not a subword of any $n$-branch word.

Choose any $x \in X$ with $x([0, n)) = w$. Define $m$ to be the minimal integer greater than $n$ so that there exists $n \leq i < m$ for which $x([i-n, i)) = x([m-n, m))$, i.e. the first place, when moving to the right from $x(0)$, where an $n$-letter word appears for the second time. Choose such an i, and suppose $i > n$. Then $x(i-n-1) \neq x(m-n-1)$, since otherwise $x([i-n-1, i-1)) = x([m-n-1, m-1))$, violating minimality of $m$. This would imply that $x([i-n, i))$ is a left-special word, and by minimality of $m$, that $x([0,i))$ is a word ending with a word in $LS_n(X)$ with no repeated $n$-letter subwords. We could then extend $x([0,i))$ to the left to create a maximal such word $x([j,i))$, which is an $n$-left branch word by definition. This $n$-left branch word would contain $w = x([0,n))$ as a subword, a contradiction.

Therefore $i = n$, i.e., $x([0, m))$ begins and ends with $w$. If $x([0, m))$ contained any words in $LS_n(X)$, then just as before we could construct an $n$-left branch word containing $w$, a contradiction. We have then shown that $x([0,m))$ begins and ends with $w$ and contains no subwords in $LS_n(X)$. Therefore, the right-most occurrence of $w$ in $x([0,m))$ 
forces letters to the left until the left-most occurrence of $w$, and this continues indefinitely. In other words, every $y \in X$ with $y([0,m)) = x([0,m))$ in fact has $y((-\infty, m))$ periodic with period $m-n$.

A similar argument shows that $x([0,m))$ cannot contain any words in $RS_n(X)$ either; if $j \geq 0$ were minimal so that $x([j, m))$ begins with a word in $RS_n(X)$, then $x([j,m))$ could be extended to the right to create an $n$-right branch word containing $x([m-n, m)) = w$, a contradiction. So $x([0,m))$ contains no words in $RS_n(X)$, meaning that the left-most occurrence of $w$ forces letters to the right until the right-most occurrence. It follows that if $y \in X$ satisfies $y([0,m)) = x([0,m))$, then $y([0, \infty))$ is periodic with period $m-n$.

Altogether, what we have shown is that every $y \in X$ with $y([0, m)) = x([0, m))$ is a periodic point with period $m-n$ coming from biinfinite repetition of $x([0, m-n))$. Therefore, $x$ is an isolated periodic point, verifying the claim that every $w \in L_n(X)$ which is not the subword of any $n$-branch word must be a subword of an isolated periodic point.

Now, choose any $\phi_1,\phi_2 \in \textrm{Aut}^{(FIP)}(X, \sigma)$ with range $\lfloor (n-1)/2 \rfloor$ and inducing block maps $\Phi_1$ and $\Phi_2$, and assume that $\Phi_{1}(v) = \Phi_{2}(v)$ for all $n$-branch words $v$.
Define $n' = 2 \lfloor (n-1)/2 \rfloor + 1$, so that $\Phi_1$ and $\Phi_2$ have domain $\mathcal{A}^{n'}$; clearly $n' \leq n$. Since $\phi_1$ and $\phi_2$ fix isolated periodic points, for all $n'$-letter subwords $u$ of such points, $\Phi_1(u) = \Phi_2(u)$. Choose any $w \in L_{n'}(X)$ which is not a subword of such a point; since $n' \leq n$, by the above it is a subword of an $n$-branch word $v$. 
Now, since $\Phi_{1}(v) = \Phi_{2}(v)$, and $w$ is a subword of $v$, $\Phi_1(w) = \Phi_2(w)$. We now know that $\Phi_1$ and $\Phi_2$ agree on all words in $L_{n'}(X)$, so $\Phi_1 = \Phi_2$, meaning that $\phi_1 = \phi_2$.\\


By (3), the number of automorphisms of range $\lfloor (n-1)/2 \rfloor$ which fix isolated periodic points is bounded from above by the number of possible choices for $\Phi(w)$ for all $n$-branch words $w$. Each $\Phi(w)$ is determined by the length of $w$ (which is independent of $\phi$) and some word of length $|w| - 2\lfloor (n-1)/2 \rfloor \leq |w| - n + 2$. By (2), the number of such words is less than or equal to $c_{1+c_n(X)}(X)$. By (1), the number of
$w$ is bounded by $2|A| (c_{n+1}(X) - c_n(X))$, completing the proof.

\end{proof}

Corollary~\ref{autbd} now follows immediately since an infinite transitive subshift $X$ has no isolated points.

We will also need the following technical lemma, which will allow us to use low complexity along a subsequence to prove existence of a (possibly sparser) subsequence where both complexity and first difference of complexity are small.
\begin{lemma}\label{diff}
For any sequences of positive reals $f(n)$ and $g(n)$ where
\[
\liminf f(n) - \sum_{i=1}^n g(i) = -\infty,
\]
there exist infinitely many values of $n$ where $f(n) < \sum_{i=1}^n g(i)$ and $f(n) - f(n-1) < g(n)$.
\end{lemma}

\begin{proof}
We first note that the hypothesis immediately implies that there exist infinitely many $n$ where $f(n) - f(n-1) < g(n)$; if not, then there would be $N$ where $f(n) - f(n-1) \geq g(n)$ for all $n > N$, meaning that
\[
f(n) = f(N) + \sum_{i=N+1}^{n} f(i) - f(i-1) \geq
(f(N) - \sum_{i=1}^{N} g(i)) + \sum_{i=1}^n g(i) \textrm{ for all } n > N,
\]
a contradiction to the assumption.

We now break into two cases. First, suppose that there exists $N$ so that $f(n) < \sum_{i=1}^n g(i)$ for $n > N$. Combining with the previous paragraph then yields the conclusion of the theorem.

Now, suppose that there exist infinitely many $n$ where $f(n) \geq \sum_{i = 1}^{n} g(i)$. The hypothesis of the theorem implies that there are also infinitely many $n$ where $f(n) < \sum_{i=1}^n g(i)$. This implies that there are infinitely many $n$ where
$f(n-1) \geq \sum_{i=1}^{n-1} g(i)$ and $f(n) < \sum_{i=1}^{n} g(i)$ (i.e. the sign of the inequality `switches infinitely many times').
But for any such $n$,
\[
f(n) - f(n-1) < \sum_{i=1}^{n} g(i) - \sum_{i=1}^{n-1} g(i) = g(n),
\]
completing the proof.

\end{proof}

We are now prepared to prove Theorems~\ref{mainthm2}-\ref{mainthm2d}. We briefly note that if $X$ is finite, then $\aut$ and $\grp$ are finite, and the conclusions of these theorems trivially hold. We therefore treat only the case where $X$ is infinite in all proofs.

\begin{proof}[Proof of Theorem~\ref{mainthm2}]

Choose any infinite transitive subshift $X$ with \newline $\liminf \frac{\log (c_n(X)/n)}{\log \log \log n} = 0$, and take $\epsilon > 0$ where $5\epsilon$ is less than the constant $c$ from Theorem~\ref{ST}. We first claim that
\begin{equation}\label{neginf}
\liminf c_n(X) - \sum_{i = 2}^n \lfloor (\log \log i)^{\epsilon} \rfloor = -\infty.
\end{equation}
To see this, by assumption, there are infinitely many $n$ where $c_n(X) < n (\log \log n)^{\epsilon/2}$, which is less than
$(n/3) (\log \log (n/2))^{\epsilon}$ for large enough $n$. Also, $\sum_{i = 2}^n \lfloor (\log \log i)^{\epsilon} \rfloor \geq
\sum_{i = \lceil n/2 \rceil}^n \lfloor (\log \log i)^{\epsilon} \rfloor \geq
(n/2) \lfloor (\log \log (n/2))^{\epsilon} \rfloor$, and so for infinitely many $n$, $c_n(X) - \sum_{i = 2}^n \lfloor (\log \log i)^{\epsilon} \rfloor$ is less than
\[
(n/3) (\log \log (n/2))^{\epsilon} - (n/2) \lfloor (\log \log (n/2))^{\epsilon} \rfloor \\ \leq n/2 - (n/6) (\log \log (n/2))^{\epsilon},
\]
which approaches $-\infty$, verifying (\ref{neginf}).
We now apply Lemma~\ref{diff}, and see that there exist infinitely many $n$ for which 
\begin{equation}\label{simult}
c_n(X) < \sum_{i = 2}^n \lfloor (\log \log i)^{\epsilon} \rfloor <
n \lfloor (\log \log n)^{\epsilon} \rfloor \textrm{ and } c_n(X) - c_{n-1}(X) < \lfloor (\log \log n)^{\epsilon} \rfloor.
\end{equation}

Now, by Corollary~\ref{autbd}, for any $n$ satisfying (\ref{simult}), $|\textrm{Aut}_{\lfloor (n-1)/2 \rfloor}(X, \sigma)|$ is bounded from above by
\[
(c_{1+c_n(X)}(X))^{2|A| (c_{n+1}(X) - c_n(X))} \leq (c_{n\lfloor (\log \log n)^{\epsilon}\rfloor}(X))^{2|A| (\log \log n)^{\epsilon}}.
\]
By subadditivity,
\[
c_{n\lfloor (\log \log n)^{\epsilon}\rfloor}(X) \leq c_n(X)^{\lfloor (\log \log n)^{\epsilon} \rfloor} < \left((n\log \log n)^{\epsilon}\right)^{(\log \log n)^{\epsilon}}
\]
\[
\leq \left(n^{2}\right)^{(\log \log n)^{\epsilon}}
= n^{2 (\log \log n)^{\epsilon}}.
\]
Therefore,
\begin{equation}\label{autbdeqn}
|\textrm{Aut}_{\lfloor (n-1)/2 \rfloor}(X, \sigma)| \leq \left(n^{2 (\log \log n)^{\epsilon}}\right)^{2|A| (\log \log n)^{\epsilon}} = n^{4|A| (\log \log n)^{2\epsilon}}.
\end{equation}
holds for any of the (infinitely many) $n$ satisfying (\ref{simult}).

Now, choose any finite subset $S$ of $\grp$ and let $S^{\prime}$ be a finite set in $\aut$ whose image under the quotient map $\aut \to \grp$ is the set $S$. Suppose that $k$ is large enough that all automorphisms in $S'$ and their inverses have range $k$. Then by additivity of ranges of automorphisms under composition, any composition of $m$ elements of $S'$ is an automorphism of range $km$.

Then for any $n$ for which (\ref{simult}) holds, (\ref{autbdeqn}) implies that the number $B_{\lfloor (n-1)/2 \rfloor/k}(S')$ of compositions of at most $\lfloor (n-1)/2 \rfloor/k$ elements of $S'$ satisfies
\[
|B_{\lfloor (n-1)/2 \rfloor/k}(S')| \leq n^{4|A| (\log \log n)^{2\epsilon}}.
\]
Since (\ref{simult}) holds for infinitely many $n$, we may choose such an $n$ greater than $e^{e^3}$, $9k^2$, $e^{2e^{(8|A|/5\epsilon)^{\epsilon^{-1}}}}$, and $c^{-2}$ (here $c$ is as in Theorem~\ref{ST}). Then $\lfloor (n-1)/2 \rfloor/k > n/3k > \sqrt{n}$, and $\log \log \sqrt{n} = \log \log n - \log 2 > \sqrt{\log \log n}$ since $\log \log n > 3$, so
\[
|B_{\sqrt{n}}(S')| \leq n^{4|A| (\log \log n)^{2\epsilon}} \leq \sqrt{n}^{8|A| (\log \log \sqrt{n})^{4\epsilon}}.
\]
Since $n > e^{2e^{(8|A|/5\epsilon)^{\epsilon^{-1}}}}$, then $8|A| < 5 \epsilon (\log \log \sqrt{n})^{\epsilon}$, and so
\[
|B_{\sqrt{n}}(S')| < \sqrt{n}^{5\epsilon (\log \log \sqrt{n})^{5\epsilon}}.
\]
Finally, since $\sqrt{n} > c^{-1}$, by Theorem~\ref{ST},
$\langle S' \rangle$ is virtually nilpotent.

Therefore, $\langle S \rangle = \faktor{\langle S' \rangle}{\langle \sigma \rangle}$ is a quotient group of a virtually nilpotent group and so itself virtually nilpotent. Let $H$ be a finite index nilpotent subgroup of $\langle S \rangle$; it is finitely generated as it is a finite index subgroup of a finitely generated group. By Theorem~\ref{CK}, $\grp$ is periodic, so $H$ is also periodic. Altogether we have that $H$ is finitely generated, periodic, and nilpotent, and therefore finite, implying that $\langle S \rangle$ is finite as well. Since $S$ was an arbitrary finite subset of $\grp$, we have shown that $\grp$ is locally finite.

\end{proof}

\begin{proof}[Proof of Theorem~\ref{mainthm2b}]

Assume that the Gap Conjecture holds. We change almost nothing about the proof of Theorem~\ref{mainthm2}, but must simply change our estimates for the usage of Lemma~\ref{diff}.

Choose any infinite transitive subshift $X$ where $\liminf \frac{c_n(X)}{n^{1.25} (\log n)^{-0.5}} = 0$. We first claim that for any $\epsilon > 0$,
\begin{equation}\label{neginf2}
\liminf c_n(X) - \sum_{i = 2}^n \lfloor \epsilon i^{0.25} (\log i)^{-0.5} \rfloor = -\infty.
\end{equation}
To see this, note that by assumption, there are infinitely many $n$ where $c_n(X) < \frac{\epsilon }{3} n^{1.25} (\log n)^{-0.5}$. 
Also, $\sum_{i = 2}^n \lfloor \epsilon i^{0.25} (\log i)^{-0.5} \rfloor \geq
\sum_{i = \lceil n/2 \rceil}^n \lfloor \epsilon i^{0.25} (\log i)^{-0.5} \rfloor \geq \newline
(n/2) \lfloor \epsilon (n/2)^{0.25} (\log n)^{-0.5} \rfloor \geq
\frac{\epsilon }{2^{1.25}} n^{1.25} (\log n)^{-0.5} - \frac{n}{2}$. So, for infinitely many $n$,
\[
c_n(X) - \sum_{i = 2}^n \lfloor \epsilon i^{0.25} (\log i)^{-0.5} \rfloor <
\frac{\epsilon }{3} n^{1.25} (\log (n/2))^{-0.5} - \frac{\epsilon }{2^{1.25}} n^{1.25} (\log (n/2))^{-0.5} + \frac{n}{2}.
\]
Since this last term approaches $-\infty$, we have verified (\ref{neginf2}).
We now apply Lemma~\ref{diff}, and see that there exist infinitely many $n$ for which
\begin{multline}\label{simult2}
c_n(X) < \sum_{i = 2}^n \lfloor \epsilon i^{0.25} (\log i)^{-0.5} \rfloor \leq n \lfloor \epsilon n^{0.25} (\log n)^{-0.5} \rfloor \\ \textrm{ and }
c_n(X) - c_{n-1}(X) < \lfloor \epsilon n^{0.25} (\log n)^{-0.5} \rfloor.
\end{multline}

Now, by Corollary~\ref{autbd}, if $n$ satisfies (\ref{simult2}), $|\textrm{Aut}_{\lfloor (n-1)/2 \rfloor}(X, \sigma)|$ is bounded from above by
\[
(c_{1+c_n(X)}(X))^{2|A| (c_{n+1}(X) - c_n(X))} \leq
(c_{n \lfloor \epsilon n^{0.25} (\log n)^{-0.5} \rfloor}(X))^{2|A| \epsilon n^{0.25} (\log n)^{-0.5}}.
\]
By subadditivity,
\begin{multline*}
c_{n \lfloor \epsilon n^{0.25} (\log n)^{-0.5} \rfloor}(X) \leq c_n(X)^{\lfloor \epsilon n^{0.25} (\log n)^{-0.5} \rfloor}\\
\leq (n^{1.25})^{\epsilon n^{0.25} (\log n)^{-0.5}} < n^{2\epsilon n^{0.25} (\log n)^{-0.5}}.
\end{multline*}
Therefore, $|\textrm{Aut}_{\lfloor (n-1)/2 \rfloor}(X, \sigma)|$ is bounded from above by
\begin{equation*}
\left(n^{2\epsilon n^{0.25} (\log n)^{-0.5}}\right)^{2|A|\epsilon  n^{0.25} (\log n)^{-0.5}} =
n^{4|A| \epsilon^2 n^{0.5} (\log n)^{-1}} = e^{4|A| \epsilon^2 \sqrt{n}}.
\end{equation*}

Now, exactly as in the end of Theorem~\ref{mainthm2}, any finitely generated subgroup $H$ of $\grp$ has growth less than $e^{4|A| \epsilon^2 \sqrt{n}}$. Since $\epsilon > 0$ was arbitrary, by the Gap Conjecture $H$ must be virtually nilpotent. Exactly as in the proof of Theorem~\ref{mainthm2}, this implies that $\grp$ is locally finite.

\end{proof}

\begin{proof}[Proof of Theorem~\ref{mainthm2c}]

Assume that the Gap Conjecture holds, and choose any infinite transitive subshift $X$ where $\frac{c_n(X)}{n^{1.5} (\log n)^{-1}} \rightarrow 0$. We first claim that for any $\epsilon > 0$,
\begin{equation}\label{neginf3}
\liminf c_n(X) - \sum_{i = 2}^n \lfloor \epsilon i^{0.5} (\log i)^{-1}\rfloor = -\infty.
\end{equation}
Again, by assumption, there are infinitely many $n$ where $c_n(X) < \frac{\epsilon}{3} n^{1.5} (\log n)^{-1}$. Also, $\sum_{i = 2}^n \lfloor \epsilon i^{0.5} (\log i)^{-1} \rfloor \geq \sum_{i = \lceil n/2 \rceil}^n \lfloor \epsilon i^{0.5} (\log i)^{-1} \rfloor \geq (n/2) \lfloor \epsilon (n/2)^{0.5} (\log n)^{-1} \rfloor \geq \frac{\epsilon}{2^{1.5}} n^{1.5} (\log n)^{-1} - \frac{n}{2}$. So, for infinitely many $n$,
\[
c_n(X) - \sum_{i = 2}^n \epsilon \sqrt{n} (\log n)^{-1} < \frac{\epsilon}{3} n^{1.5} (\log n)^{-1} - \frac{\epsilon}{2^{1.5}} n^{1.5}
(\log n)^{-1} + \frac{n}{2}.
\]
Since this last term approaches $-\infty$, we have verified (\ref{neginf3}).
We now apply Lemma~\ref{diff}, and see that there exist infinitely many $n$ for which
\begin{equation}\label{simult3}
c_n(X) - c_{n-1}(X) < \lfloor \epsilon n^{0.5} (\log n)^{-1} \rfloor.
\end{equation}
Rather than using Lemma~\ref{diff} to bound $c_n(X)$, we simply recall that by assumption, there exists $N$ so that
\begin{equation}\label{simult3.5}
c_n(X) < \lfloor n^{1.5} \rfloor
\end{equation}
for all $n>N$. This clearly implies that $c_{1 + c_n(X)}(X) \leq (n^{1.5})^{1.5} = n^{2.25}$ for any $n>N$. By Corollary~\ref{autbd}, for any of the infinitely many $n > N$ satisfying (\ref{simult3}) and (\ref{simult3.5}), $|\textrm{Aut}_{\lfloor (n-1)/2 \rfloor}(X, \sigma)|$ is bounded from above by
\begin{equation*}
(c_{1+c_n(X)}(X))^{2|A| (c_{n+1}(X) - c_n(X))} \leq (n^{2.25})^{2|A| \epsilon n^{0.5} (\log n)^{-1}} = e^{4.5 |A| \epsilon \sqrt{n}}.
\end{equation*}

Now, exactly as in the end of Theorem~\ref{mainthm2}, any finitely generated subgroup $H$ of $\grp$ has growth less than $e^{4.5 |A| \epsilon \sqrt{n}}$. Since $\epsilon > 0$ was arbitrary, by the Gap Conjecture $H$ must be virtually nilpotent. Exactly as in the proof of Theorem~\ref{mainthm2}, this implies that $\grp$ is locally finite.

\end{proof}

\begin{proof}[Proof of Theorem~\ref{mainthm2d}]

Choose any subshift $X$ where $\frac{c_n(X)}{n^{2} (\log n)^{-1}} \rightarrow 0$, and any $\epsilon > 0$. We claim that
\begin{equation}\label{neginf4}
\liminf c_n(X) - \sum_{i = 2}^n \lfloor \epsilon i (\log i)^{-1} \rfloor = -\infty.
\end{equation}
Again, by assumption, there are infinitely many $n$ where $c_n(X) < \frac{\epsilon}{5} n^2 (\log n)^{-1}$. 
Also,
\[
\sum_{i = 2}^n \lfloor \epsilon i (\log i)^{-1} \rfloor \geq
\sum_{i = \lceil n/2 \rceil}^n \lfloor \epsilon i (\log i)^{-1} \rfloor \geq
(n/2) \lfloor (\epsilon n/2) (\log n)^{-1} \rfloor \geq \frac{\epsilon}{4} n^2 (\log n)^{-1} - \frac{n}{2}.
\]
So, for infinitely many $n$,
\[
c_n(X) - \sum_{i = 2}^n \lfloor i (\log i)^{-1} \rfloor < \frac{\epsilon}{5}  n^2 (\log n)^{-1} - \frac{\epsilon}{4} n^2 (\log n)^{-1} +\frac{n}{2}.
\]
Since this last term approaches $-\infty$, we have verified (\ref{neginf4}).
We now apply Lemma~\ref{diff}, and see that there exist infinitely many $n$ for which
\begin{equation}\label{simult4}
c_n(X) - c_{n-1}(X) < \lfloor \epsilon n (\log n)^{-1} \rfloor.
\end{equation}
Rather than using Lemma~\ref{diff} to bound $c_n(X)$, we simply recall that by assumption, there exists $N$ so that
\begin{equation}\label{simult4.5}
c_n(X) < n^2
\end{equation}
for all $n>N$. This clearly implies that $c_{1 + c_n(X)}(X) \leq (n^{2})^{2} = n^{4}$ for any $n>N$. Now, by Theorem~\ref{autbd2}, for any of the infinitely many $n > N$ satisfying (\ref{simult4}) and (\ref{simult4.5}), $\left|\textrm{Aut}^{(FIP)}_{\lfloor (n-1)/2 \rfloor}(X, \sigma)\right|$ is bounded from above by
\begin{equation*}
(c_{1+c_n(X)}(X))^{2|A| (c_{n+1}(X) - c_n(X))} \leq (n^{4})^{2|A| \epsilon n (\log n)^{-1}} = e^{8 |A| \epsilon n}.
\end{equation*}

By compactness, the set of isolated periodic points of $X$ is finite; denote this set by $\mathcal{P}$. It is straightforward to check that the set $\mathcal{P}$ is invariant under any automorphism of $(X,\sigma)$, and we may consider the homomorphism
\begin{equation*}
\begin{gathered}
\pi_{\mathcal{P}} \colon \aut \to \textnormal{Aut}(\mathcal{P},\sigma|_{\mathcal{P}})\\
\pi_{\mathcal{P}} \colon \phi \mapsto \phi|_{\mathcal{P}}.
\end{gathered}
\end{equation*}
By definition we have $\textrm{Aut}^{(FIP)}(X, \sigma) = \textrm{ker}(\pi_{\mathcal{P}})$.
Now, exactly as in the end of Theorem~\ref{mainthm2}, any finitely generated subgroup $H$ of $\textrm{Aut}^{(FIP)}(X, \sigma)$ has growth less than $e^{8 |A| \epsilon n}$. Since $\epsilon$ was arbitrary, this implies that $H$ has subexponential growth, and so is amenable. Then, every finitely generated subgroup of $\textrm{Aut}^{(FIP)}(X, \sigma)$ is amenable, implying that $\textrm{Aut}^{(FIP)}(X, \sigma)$ is amenable. 
Since $\mathcal{P}$ is finite, $\textrm{Aut}(\mathcal{P},\sigma|_{\mathcal{P}})$ is a finite group, and hence $\textrm{Aut}^{(FIP)}(X, \sigma)$ is of finite index in $\aut$.  Since $\textrm{Aut}^{(FIP)}(X, \sigma)$ is amenable, this implies $\aut$ is amenable.


\end{proof}

\section{Realizing locally finite groups as $\grp$ for low complexity}

In this section, we prove Theorem~\ref{mainthm}. We first outline the general block concatenation construction of subshifts; for an introduction, see (\cite{HK}, \cite{Pset}). It is simple to guarantee that such a subshift be minimal. The difficult part will be to engineer our subshift to have low complexity and prescribed $\grp$.

A \textbf{block concatenation subshift} is defined by an alphabet $\mathcal{A}$, sequences $(n_k)$ of positive integers, and sets
$A_k \subset \mathcal{A}^{n_k}$ with the following property: for every $k$, every $w \in A_{k+1}$ is a concatenation of $A_k$-words. (This of course implies that $n_k | n_{k+1}$ for all $k$.) We will always take $n_1 = 1$ and $A_1 = \mathcal{A}$. Given such $\mathcal{A}$, $(n_k)$, and $(A_k)$, $X$ consists of all `limits' of $A_k$-words (as $k \rightarrow \infty$); more formally, $x \in X$ if and only if for all $n$, there exists $k$ so that $x([-n,n])$ is a subword of some $A_k$-word.

We first prove some general lemmas about block concatenation subshifts. The following is well-known (\cite{HK}), but we will give a short proof for completeness.


\begin{lemma}\label{min}
If every $A_{k+1}$-word, written as a concatenation of $A_k$-words, contains each $A_k$-word at least once, then $X$ is minimal.
\end{lemma}

\begin{proof}
For every $w \in L(X)$, there exists $k$ so that $w$ is a subword of some $A_k$-word. But then $w$ is a subword of every $A_{k+1}$-word. For every $x \in X$, $x$ contains an $A_{k+1}$-word, so contains $w$. Since $x$ and $w$ were arbitrary, $X$ is minimal.
\end{proof}

By definition, for every $x \in X$ and $k \in \mathbb{N}$, $x$ can be written as a bi-infinite concatenation of
$A_k$-words. We say that $X$ is \textbf{uniquely decomposable} if this decomposition is unique for all $x \in X$. This can also be achieved through a simple assumption about repetitions of $A_k$-words.

\begin{lemma}\label{decomp}
If $(d_k)$ is an integer sequence where each $A_{k+1}$-word, written as a concatenation of $A_k$-words, begins with $d_{k+1}$ repetitions of the same $A_k$-word, ends with $d_{k+1}$ repetitions of the same $A_k$-word, and does not elsewhere contain $d_{k+1}$ repetitions of the same $A_k$-word, then $X$ is uniquely decomposable.
\end{lemma}

\begin{proof}
We prove by induction on $k$. The base case $k = 1$ simply says that locations of $A_1$-words are uniquely determined for such subshifts, which is trivial since we always take $A_1$ to be the alphabet $\mathcal{A}$. For the inductive step, assume for some $k$ that for every $x \in X$, $x$ can be uniquely decomposed into $A_k$-words. Given this decomposition, one simply searches for $2d_{k+1}$-fold concatenations of the form $v^{d_{k+1}} w^{d_{k+1}}$ (with $v,w \in A_k$ possibly equal), which can only occur with midpoint at the border between $A_{k+1}$-words. This implies that $x$ can be written in a unique way as a concatenation of $A_{k+1}$-words, completing the inductive step and the proof.
\end{proof}

Suppose $X$ is uniquely decomposable, let $k \in \mathbb{N}$, and let $\tau \colon A_{k} \to A_{k}$ be a permutation of the $A_{k}$-words. Associated to $\tau$ is a continuous shift-commuting map $\alpha_{\tau} \colon X \to \left(\mathcal{A}^{\mathbb{Z}},\sigma\right)$ defined as follows: given $x \in X$, decompose $x$ as a concatenation of $A_{k}$-words, and apply $\tau$ to each $A_{k}$-word appearing in $X$. (Note that this map is only well-defined because $(X, \sigma)$ was assumed uniquely decomposable; shift-commuting and continuity are then nearly immediate from the definition.) Written symbolically, if we have
\[
x = \ldots w_{-1}w_{0}w_{1} \ldots, \qquad w_{i} \in A_{k}
\]
then
\begin{equation}\label{alpha}
\alpha_{\tau}(x) = \ldots \tau(w_{-1})\tau(w_{0})\tau(w_{1})\ldots.
\end{equation}

Note that depending on $\tau$, $\alpha_{\tau}$ may or may not map $X$ back into $X$; if it does, then $\alpha_{\tau}$ is an automorphism of $(X,\sigma)$.







We are now prepared to define the block concatenation subshifts which will prove Theorem~\ref{mainthm}.

\begin{proof}[Proof of Theorem~\ref{mainthm}]

Choose any unbounded increasing $f$ and countable locally finite group $G$; $G$ can be written, by definition, as the union of an increasing chain of finite subgroups, i.e. there exist finite groups $H_k$ so that $H_k$ is a proper subgroup of $H_{k+1}$ for all $n$, and $G$ is the union of the $H_k$. Choose an increasing sequence $(b_k)$ of integers greater than $1$ with the property that $f(b_k) > k(|H_k|^{5|H_{k+1}| + 2|H_k|^2})$ for all $k$.

Our technique is somewhat similar to that of \cite{BLR}, where a subshift $X$ was constructed for which the additive group of rationals embeds into $\aut$, in that we will construct, for every $k$, automorphisms defined by their action on the set $A_k$, and then show that every automorphism of $X$ can be realized in this way. Specifically, for each $k$, we will define a group of permutations of the words in $A_k$ which is isomorphic to $H_k$. We will then show that for any $k$, each $A_k$-permutation induces an $A_{k+1}$-permutation by coordinatewise application to $A_k$-words, in a manner that is compatible with the containment of $H_k$ as a subgroup of $H_{k+1}$.

We begin with some notation. For every $k$, fix an ordering $\{h_i^{(k)}\}_{i=1}^{|H_k|}$ of the elements of $H_k$ with
$h_1^{(k)} = \{\textrm{id}\}$. Write $q_k = |H_k|/|H_{k-1}|$, and choose any set $\{r_i^{(k)}\}_{i = 1}^{q_k}$ of representatives for right cosets of
$H_{k-1}$ in $H_{k}$, i.e.
\[
H_{k} = \bigcup_{i=1}^{q_k} H_{k-1} r_i^{(k)}.
\]
Without loss of generality, we will always take $r_1^{(k)} = \{\textrm{id}\}$, the identity element of $G$ (and of all $H_k$).
We will now recursively define the sets $A_k$ and lengths $n_k$. In our construction, $|A_k| = |H_k|$ for every $k$.

The $k = 1$ case is simple; for every $g \in H_1$, define a symbol $w_g^{(1)}$, and define the alphabet $A_1 = \{w_g^{(1)} \ : \ g \in H_1\}$; clearly $|A_1| = |H_1|$.


We now define $A_{k+1}$ for $k \geq 1$, assuming that $A_k = \{w_g^{(k)} \ : \ g \in H_k\}$ has been defined already.
Informally, the idea is that we will define $w_g^{(k+1)}$ for every $g \in H_{k+1}$ by assigning a different `template' concatenation of $A_k$-words to each coset representative $r_i^{(k+1)}$, and then permuting the $A_k$-words $\{w_g^{(k)}\}_{g \in H_k}$ in those templates by multiplying on the left by a properly chosen $g'$ in the subscripts.

More formally, for every $g \in H_{k+1}$, we first write $g = g' r_i^{(k+1)}$ for some $g' \in H_k$ and $1 \leq i \leq q_{k+1}$. We then define
\begin{multline*}
w_{g}^{(k+1)} = w_{g' r_i^{(k+1)}}^{(k+1)} = \Big(w_{g' h_2^{(k)}}^{(k)}\Big)^{2b_{k+1} q_{k+1}}\\
\left(
\Big(w_{g' h_1^{(k)}}^{(k)}\Big)^{b_{k+1}} \Big(w_{g' h_2^{(k)}}^{(k)}\Big)^{b_{k+1}} \Big|
\ldots \Big|
\Big(w_{g' h_{|H_k|-1}^{(k)}}^{(k)}\Big)^{b_{k+1}} \Big(w_{g' h_{|H_k|}^{(1)}}^{(1)}\Big)^{b_{k+1}}\right)\\
\Big(w_{g' h_{1}^{(k)}}^{(k)}\Big)^{ib_{k+1}} \Big(w_{g' h_{2}^{(k)}}^{(k)}\Big)^{b_{k+1}(3q_{k+1}-i)}.
\end{multline*}
Here, the concatenation inside the largest parentheses is a list of all pairs of the form
$\Big(w_{g' h_a^{(k)}}^{(k)}\Big)^{b_{k+1}} \Big(w_{g' h_b^{(k)}}^{(k)}\Big)^{b_{k+1}}$, with pairs $(a,b)$ listed in lexicographic order.
We then define $A_{k+1} = \{w_g^{(k+1)} \ : \ g \in H_{k+1}\}$, and note that all $A_{k+1}$-words begin and end with
$2b_{k+1} q_{k+1}$-fold repetitions of an $A_k$-word, and contain no other such repetitions. Note that the length of all $A_{k+1}$-words is
\begin{equation}\label{lengthbd}
n_{k+1} = b_{k+1} n_k(2|A_k|^2 + 5q_{k+1}) > b_{k+1} q_{k+1}.
\end{equation}

Recursively, this defines $n_k$ and $A_k$ for all $k$, and so an associated block concatenation subshift $X$. We here note a few properties of $X$ which will be useful later.

\begin{itemize}
\item $X$ is minimal by Lemma~\ref{min}
\item $X$ is uniquely decomposable by Lemma~\ref{decomp} (with $d_k = 2b_k q_k$)
\item Concatenations of the form $uvw$ with $u,v,w \in A_k$, $u \neq v$, and $v \neq w$ never appear in points of $X$, and that every other concatenation $uvw$ with $u = v$ or $v = w$ appears within every $A_{k+1}$-word.
\end{itemize}


We now wish to bound the complexity of $X$ from above to show that $\frac{c_n(X)}{nf(n)} \rightarrow 0$.
Choose any length $n$; there exists $k$ so that $n_k \leq n < n_{k+1}$. We first treat the case where
$n \in [n_k, b_{k+1} n_k)$. All points of $X$ are concatenations of $A_k$-words, and by definition of $A_{k+1}$,
each $A_k$-word is repeated some number of times which is a multiple of $b_{k+1}$. Therefore, since $n < b_{k+1} n_k$,
any word $w \in L_n(X)$ is of the form $s qq \ldots qq rr \ldots rr p$, where $q,r$ are $A_k$-words, $s$ is a suffix of $q$, and $p$ is a prefix of $r$. Then $w$ is determined completely by the location at which the transition from $q$ to $r$ occurs, and the choices of $q$ and $r$. (The case where $w$ has no transition is still included here by just taking $q = r$.) So, $c_n(X) \leq n |A_k|^2 = n|H_k|^2$.
Since $|H_k|^2 < f(b_k)/k < f(n_k)/k$ and $n \geq n_k$, we have $c_n(X) \leq n f(n_k)/k \leq n f(n)/k$.

Now consider the case where $n \in [b_{k+1} n_k, n_{k+1})$. Any word $w \in L_n(X)$ must be of the form
$s u_1^{b_{k+1}} u_2^{b_{k+1}} \ldots u_{j-1}^{b_{k+1}} p$, where $u_1, \ldots u_{j-1} \in A_k$,
$s$ is a suffix of some word $u_0^{b_{k+1}}$ for $u_0 \in A_k$, and $p$ is a prefix of some word $u_j^{b_{k+1}}$ for $u_j \in A_k$.
By (\ref{lengthbd}), $\frac{n_{k+1}}{b_{k+1} n_k} < 5q_{k+1} + 2|A_k|^2$, 
and so $j < 5|H_{k+1}| + 2|H_k|^2$.
Clearly, $w$ is determined by the words $u_0, \ldots, u_j$ and the length of $s$, so
$c_n(X) \leq b_{k+1} n_k |A_k|^{j+1} = b_{k+1} n_k |H_k|^{j+1}$. Since $j < 5|H_{k+1}| + 2|H_k|^2$,
\[
|H_k|^{j+1} \leq  |H_k|^{5|H_{k+1}| + 2|H_k|^2} < f(b_{k})/k < f(b_{k+1} n_k)/k.
\]
Since $n \geq b_{k+1} n_k$, this implies that $c_n(X) \leq n f(b_{k+1} n_k)/k \leq n f(n)/k$.

We've shown that for all $n \geq n_k$, $c_n(X) \leq nf(n)/k$, and so $\frac{c_n(X)}{nf(n)} \rightarrow 0$.
It remains only to show that $\grp$ is isomorphic to $G$.



For any $k$ and $h \in H_k$, define the permutation $\pi_{k,h}$ of the set $\{w_g^{(k)}\}$ of $A_k$-words by left multiplication of the subscript, i.e. $\pi_{k,h}(w_g^{(k)}) = w_{hg}^{(k)}$. In a slight abuse of notation, we also define $\pi_{k,h}$ to act on concatenations of $A_k$-words by `coordinatewise' application, i.e. if $w_1, \ldots, w_m \in A_k$, $\pi_{k,h}(w_1 \ldots w_m) :=
\pi_{k,h}(w_1) \ldots \pi_{k,h}(w_m)$.

\begin{lemma}\label{consist}
For any $k \ge 1$, $h \in H_k$, and $w \in A_{k+1}$, $\pi_{k,h}(w) = \pi_{k+1,h}(w)$.
\end{lemma}

\begin{proof}
This comes from the definition of $A_{k+1}$. Informally, it's due to the fact that a left multiplication by any $h \in H_k$ in the subscript of an $A_{k+1}$-word $w_{g' r_i^{(k+1)}}^{(k+1)}$ passes through to left multiplications of all subscripts of the component $A_k$-words.

More formally, choose any $k \ge 1$, $h \in H_k$ and $g \in H_{k+1}$; we can write $g = g' r_i^{(k+1)}$ for some $g^{\prime} \in H_k$ and
$1 \leq i \leq q_{k+1}$. Then,
\begin{multline*}
\pi_{k,h}\left(w_{g}^{(k+1)}\right) = \pi_{k,h}\left(w_{g' r_i^{(k+1)}}^{(k+1)}\right)
= \pi_{k,h}\Bigg(
\Big(w_{g' h_2^{(k)}}^{(k)}\Big)^{2b_{k+1} q_{k+1}}\\
\left(
\Big(w_{g' h_1^{(k)}}^{(k)}\Big)^{b_{k+1}} \Big(w_{g' h_2^{(k)}}^{(k)}\Big)^{b_{k+1}} \Big|
\ldots \Big|
\Big(w_{g' h_{s_k-1}^{(k)}}^{(k)}\Big)^{b_{k+1}} \Big(w_{g' h_{s_k}^{(1)}}^{(1)}\Big)^{b_{k+1}}\right)\\
\Big(w_{g' h_{1}^{(k)}}^{(k)}\Big)^{ib_{k+1}} \Big(w_{g' h_{2}^{(k)}}^{(k)}\Big)^{b_{k+1}(3q_{k+1}-i)}
\Bigg)
= \Big(w_{hg' h_2^{(k)}}^{(k)}\Big)^{2b_{k+1} q_{k+1}}\\
\left(
\Big(w_{hg' h_1^{(k)}}^{(k)}\Big)^{b_{k+1}} \Big(w_{hg' h_2^{(k)}}^{(k)}\Big)^{b_{k+1}} \Big|
\ldots \Big|
\Big(w_{hg' h_{s_k-1}^{(k)}}^{(k)}\Big)^{b_{k+1}} \Big(w_{hg' h_{s_k}^{(1)}}^{(1)}\Big)^{b_{k+1}}\right)\\
\Big(w_{hg' h_{1}^{(k)}}^{(k)}\Big)^{ib_{k+1}} \Big(w_{hg' h_{2}^{(k)}}^{(k)}\Big)^{b_{k+1}(3q_{k+1}-i)}\\
= w_{hg' r_i^{(k+1)}}^{(k+1)} = \pi_{k+1,h}\left(w_{g' r_i^{(k+1)}}^{(k+1)}\right) = \pi_{k+1,h}\left(w_{g}^{(k+1)}\right).\\
\end{multline*}

\end{proof}


In particular, by induction, Lemma~\ref{consist} implies that for any $k < m$, $\pi_{k,h}$ induces the permutation $\pi_{m,h}$ of
$A_m$-words by coordinatewise action. By passing to limits, we see that in fact $\pi_{k,h}$ induces a self-bijection of $X$, which we denote by $\phi_{k,h}$. Since $X$ is uniquely decomposable, $\phi_{k,h}$ is continuous and shift-commuting (in fact it
is one of the maps $\alpha_{\tau}$ defined in (\ref{alpha})), 
and so is in $\textrm{Aut}(X, \sigma)$.

For every $k$, the group $\{\phi_{k,h}\}_{h \in H_k}$ is clearly isomorphic to $H_k$ itself (since $\phi_{k,h} \circ \phi_{k,h'} = \phi_{k,hh'}$). The collection $\{\phi_{k,h}\}_{k \in \mathbb{N}, h \in H_k}$ then forms a subgroup of $\aut$, which we denote by $G_{\phi}$, and it follows from the above that $G_{\phi}$ is isomorphic to $G$.
We now claim that even after quotienting out by the subgroup generated by the shift, this is still true.


\begin{lemma}\label{mod}
Let $\rho \colon \aut \to \grp$ denote the quotient map. Then $\rho(G_{\phi})$ is isomorphic to $G$.
\end{lemma}

\begin{proof}

It's enough to show that $G_{\phi} \cap \ker \rho = \textnormal{id}$. Thus it's enough to show that if $\phi_{k,h}=\sigma^{m}$ for some $k,h,m$, then $m=0$. Suppose then that $\sigma^{m} = \phi_{k,h}$. Choose $k' \geq k$ so that
$|m| < n_{k'}$. Then, by Lemma~\ref{consist}, $\phi_{k,h} = \phi_{k', h}$, and hence $\sigma^{m} = \phi_{k',h}$.
Finally, we note that for any $x \in X$, $\phi_{k',h}(x)$ has $A_{k'}$-words in the same locations as $x$. 
Since $\sigma^m(x) = \phi_{k',h}(x)$ and $|m| < n_{k'}$, the only way for this to happen is if $m=0$, completing the proof.
\end{proof}

Finally, we must show that under the quotient map $\rho \colon \aut \to \grp$, the image of $G_{\phi}$ is all of $\grp$; in other words, that every automorphism of $X$ can be written as $\sigma^j \phi_{k,h}$ for some $j \in \mathbb{Z}$, $k \in \mathbb{N}$, and $h \in H_k$.

\

We need a technical definition; we say that $\phi \in \textrm{Aut}(X, \sigma)$ \textbf{preserves locations of $A_k$-words} if, for all $x \in X$ and $m < n$, if $x([m,m+n_k))$ is an $A_k$-word, then $(\phi x)([m,m+n_k))$ is an $A_k$-word. (Note that preserving locations of $A_k$-words clearly implies preserving locations of $A_j$-words for any $j < k$.) We say that $\phi$ simply \textbf{preserves locations} if it preserves locations of $A_k$-words for all $k$.

It's clear that for every $\phi \in \aut$ of range $n_{k}$ 
and every $x \in X$, there exists a shift $i$ with $|i| \leq n_k/2$ so that $(\phi \circ \sigma^i)(x)$ has $A_k$-words at the same locations as $x$. In theory though, this is weaker than preserving locations of $A_k$-words, as $i$ could depend on $x$. For our examples, we can show that this is not possible as long as $k$ is large enough.

\begin{lemma}\label{pres0}
If $\phi \in \aut$ has range $n_k$, and for some $x \in X$, $\phi x$ has $A_k$-words at the same locations as $x$, then $\phi$ preserves locations of $A_k$-words.
\end{lemma}

\begin{proof}

Choose any $k$, $\phi$ with range $n_k$ and inducing block map $\Phi$, $x \in X$, and suppose that $\phi x$ has $A_k$-words at the same locations as $x$. By shifting $x$, we may assume without loss of generality that $x([0, n_{k+1})) \in A_{k+1}$. 
Then $x([0, n_{k+1}))$, as an $A_{k+1}$-word, contains all concatenations of three $A_k$-words which are in $L(X)$, i.e., for all
$u,v,w \in A_{k}$ such that $uvw \in L(X)$ (i.e. $u=v$ or $v=w$), there exists $n_k \leq i < n_{k+1} - 2n_k$ so that $x([i-n_k, i+2n_k)) = uvw$. This implies that $\Phi(uvw) = (\phi x)([i, i+n_k)) \in A_k$.

For any $y \in X$ and $j \in \mathbb{Z}$, if $y([j, j+n_k)) \in A_k$, then $y([j-n_k, j+2n_k)) = uvw$ for some $u,v,w \in A_k$, implying that $(\phi y)([j, j+n_k)) = \Phi(uvw) \in A_k$. 
Since $y$ was arbitrary, $\phi$ preserves locations of $A_k$-words.

\end{proof}

We now wish to show that all automorphisms of $X$ preserve locations up to a shift.

\begin{lemma}\label{pres1}
For every $\phi \in \textrm{Aut}(X, \sigma)$, there exists $i$ so that $\sigma^i \circ \phi$ preserves locations.
\end{lemma}

\begin{proof}
Choose any $\phi \in \textrm{Aut}(X, \sigma)$, and choose $k$ so that both $\phi$ and $\phi^{-1}$ have range $n_k/2$.
Choose any $x \in X$ with $x([0, n_{k+1})) \in A_{k+1}$. There clearly exists $i$ with $|i| < n_k/2$ so that
$(\phi \circ \sigma^i x)([mn_k, (m+1)n_k)) \in A_k$ for all $m$. Write $\phi' := \phi \circ \sigma^i$. By additivity of ranges under composition, $\phi'$ and $\phi'^{-1}$ have range $n_k$, and by Lemma~\ref{pres0}, $\phi'$ preserves locations of $A_k$-words.

Now, consider any $x \in X$ for which $x([-n_{k+1}, 0))$ and $x([0, n_{k+1}))$ are both in $A_{k+1}$, and $x([-n_{k+1}, 0))$
ends with the same $2b_{k+1} c_{k+1}$-fold repetition of an $A_k$-word that $x([0, n_{k+1}))$ begins with. Then
$x([-2b_{k+1} c_{k+1} n_k, 2b_{k+1} c_{k+1} n_k)) = w^{4b_{k+1} c_{k+1}}$ for some $w \in A_k$. 
Since $\phi$ preserves locations of $A_k$-words and has range less than $n_k$,
$(\phi' x)([(-2b_{k+1} c_{k+1} + 1)n_k, (2b_{k+1} c_{k+1} - 1)n_k)) = v^{4b_{k+1} c_{k+1} - 2}$ for some $v \in A_k$. However, the only
$(4b_{k+1} c_{k+1} - 2)$-fold repetitions of $A_k$-words within points of $x$ are within the $(4b_{k+1} c_{k+1})$-fold repetitions occurring across the boundary of some pairs of $A_{k+1}$-words. Therefore, there exists some $j \in \{0, \pm n_k\}$ for which
$((\phi' \circ \sigma^j) x)([-2b_{k+1} c_{k+1} n_k, 2b_{k+1} c_{k+1} n_k]) = v^{4b_{k+1} c_{k+1}}$, which implies that
$((\phi' \circ \sigma^j) x)([-n_{k+1}, 0))$ and $((\phi' \circ \sigma^j) x)([0, n_{k+1}))$ are $A_{k+1}$-words, and so that
$(\phi' \circ \sigma^j) x$ has $A_{k+1}$-words in the same locations as $x$. Define $\phi'' := \phi' \circ \sigma^j$; since $\phi', \phi'^{-1}$ had range $n_k$ and $|j| \leq n_k$, $\phi''$ and $\phi''^{-1}$ have range $2n_k$ by additivity of ranges under composition. Since $2n_k < n_{k+1}$, Lemma~\ref{pres0} implies that $\phi''$
 preserves locations of $A_{k+1}$-words. We claim that in fact $\phi''$ preserves locations for all $A_m$-words for $m > k$, which will complete the proof since $\phi'' = \phi \circ \sigma^{i+j}$. We prove by induction on $m$. The base case $m = k+1$ is completed, so we assume that $m > k$ and that $\phi''$ preserves locations of $A_m$-words.

Choose two $A_{m}$-words $y, z$ which are not equal, but agree on their first and last $2n_m$ letters (for instance,
$w_{\textrm{id}}^{(m)}$ and $w_{r_2^{(m)}}^{(m)}$ would work.) Choose $x \in X$ so that $x([-n_{m+1}, 0))$ and $x([0, n_{m+1}))$ are both in $A_{m+1}$, $x([-n_{m+1}, 0))$ ends with $y^{2b_{m+1} c_{m+1}}$, and $x([0, n_{m+1}))$ begins with $z^{2b_{m+1} c_{m+1}}$.

Since $\phi''$ preserves locations of $A_{m}$-words, all words
$(\phi'' x)([in_{m}, (i+1)n_{m}))$ for $-2b_{m+1} c_{m+1} \leq i < 2b_{m+1} c_{m+1}$ are $A_{m}$-words.
Since $\phi''$ has range $2n_k$, each $(\phi'' x)([in_{m}, (i+1)n_{m}))$ depends only on
$x(in_{m} - 2n_k, (i+1)n_{m} + 2n_k))$. Since $y$ and $z$ agree on their first and last $2n_m \geq 2n_k$ letters and
\[
x([-2b_{m+1} c_{m+1} n_{m}, 2b_{m+1} c_{m+1} n_{m})) = y^{2b_{m+1} c_{m+1}} z^{2b_{m+1} c_{m+1}},
\]
$(\phi'' x)([in_{m}, (i+1)n_{m}))$ is the same for $-2b_{m+1} c_{m+1} < i < 0$ (call this $A_{m}$-word $a$) and for
$0 \leq i < 2b_{m+1} c_{m+1} - 1$ (call this $A_{m}$-word $b$).

If $a = b$, then since $\phi''^{-1}$ has range $2n_k < n_{m}$, it would have to be the case that $y = z$
(formally, if $\Psi$ is an inducing block map for $\phi''^{-1}$, then since $a = (\phi'' x)([-2n_{m} - 2n_k, -n_{m} + 2n_k) = (\phi'' x)([n_{m} - 2n_k, 2n_{m} + 2n_k) = b$, it must be the case that $y = \Psi(a) = x([-2n_{m}, -n_{m}) = x([n_{m}, 2n_{m}) = \Psi(b) = z$.) This is a contradiction, so $a \neq b$. Now, we know that
$(\phi'' x)([(-2b_{m+1} c_{m+1} + 1)n_{m}, (2b_{m+1} c_{m+1} - 1)n_{m})) = a^{2b_{m+1} c_{m+1} - 1} b^{2b_{m+1} c_{m+1} - 1}$
for $a \neq b \in A_{m}$, and the only occurrence of such a word is with midpoint at the border between $A_{m+1}$-words. Therefore,
$(\phi'' x)([-n_{m+1}, 0))$ and $(\phi'' x)([0, n_{m+1}))$ are $A_{m+1}$-words, and so $\phi'' x$ has $A_{m+1}$-words at the same locations as $x$. Since $\phi''$ has range $2n_k < n_{m+1}$, $\phi''$ preserves locations of $A_{m+1}$-words by Lemma~\ref{pres0}. This completes the induction and the proof.

\end{proof}

The following lemma now completes the proof of Theorem~\ref{mainthm}.

\begin{lemma}\label{pres2}
If $\phi \in \textrm{Aut}(X, \sigma)$ preserves locations, then there exist $k \in \mathbb{N}$ and $h \in H_k$ so that $\phi = \phi_{k,h}$.
\end{lemma}

\begin{proof}

Assume that $\phi \in \aut$ preserves locations, fix $k$ so that $\phi$ has range $n_k$, and let $\Phi$ be an inducing block map for $\phi$. Choose any $x \in X$ with $x([0, n_k)) \in A_k$. Then $x([in_k, (i+1)n_k)) \in A_k$ for all $i$, and by assumption,
$(\phi x)([in_k, (i+1)n_k)) \in A_k$ for all $i$ as well. Then, exactly as in the proof of Lemma~\ref{pres0},
for all $u,v,w \in A_k$ with $u=v$ or $v=w$, $\Phi(uvw) \in A_k$. 
We will prove that $\phi$ is equal to some $\phi_{k,h}$ in two steps. First, we will show that $\Phi(uvw)$ depends only on $v$, implying that $\phi$ is induced by a permutation of $A_k$-words in the sense of (\ref{alpha}). Then, we show that the only such permutations which send $A_{k+1}$-words to $A_{k+1}$-words are of the form $\pi_{k,h}$.

To see that $\Phi(uvw)$ depends only on $v$, choose any $(u,v,w) \neq (u', v, w') \in S$.
By its definition, the $A_{k+1}$-word $w_{\textrm{id}}^{(k+1)}$ contains both $uvw$ and $u'vw'$ somewhere in its
first $(2b_{k+1} c_{k+1} + b_{k+1}|A_k|^2)$ concatenated $A_k$-words, say that
$w_{\textrm{id}}^{(k+1)}([(p-1)n_k, (p+2)n_k)) = uvw$ and $w_{\textrm{id}}^{(k+1)}([(q-1)n_k, (q+2)n_k)) = u'vw'$ for some
$p,q \leq 2b_{k+1} c_{k+1} + b_{k+1}|A_k|^2$.

Choose any $x \in X$ with $x([0, n_{k+1})) = w_{\textrm{id}}^{(k+1)}$. Since $\phi$ preserves locations of $A_{k+1}$-words, $(\phi x)([0, n_{k+1}))$ is some $A_{k+1}$-word $w_{h r_i^{(k+1)}}^{(k+1)} = \pi_{k,h} w_{r_i^{(k+1)}}^{(k+1)}$ with $h \in H_k$ and $1 \leq i \leq q_{k+1}$. Note that $w_{r_i^{(k+1)}}^{(k+1)}$ begins with the same initial $2b_{k+1} c_{k+1} + b_{k+1}|A_k|^2$ concatenated $A_k$-words as $w_{\textrm{id}}^{(k+1)}$, and so contains $v$ starting at locations $pn_k$ and $qn_k$. But $\pi_{k,h}$ is just a permutation of $A_k$-words, and so $w_{h r_i^{(k+1)}}^{(k+1)}$ contains the same $A_k$-word $\pi_{k,h}(v)$ at those locations. Therefore, $\Phi(uvw) =
\Phi(u'vw') = \pi_{k,h}(v)$. Since $(u,v,w)$ and $(u', v, w')$ were arbitrary, we've shown that $\Phi(uvw)$ depends only on $v$,
i.e. there is a permutation $\tau$ of $A_k$-words so that $\phi = \alpha_{\tau}$ as in (\ref{alpha}).

Consider the image of the $A_{k+1}$-word $w_{\textrm{id}}^{(k+1)}$ under (coordinatewise application of) $\phi = \alpha_\tau$. It must be another $A_{k+1}$-word since $\phi$ preserves locations, call it $w_{h r_i^{(k+1)}}^{(k+1)}$ for $h \in H_k$ and $1 \leq i \leq q_{k+1}$.
We note that $w_{h_2^{(k)}}^{(k)}$ occurs $b_{k+1} (5q_{k+1} + |A_k| - 1)$ times in the decomposition of $w_{\textrm{id}}^{(k+1)}$ into $A_{k}$-words, and so some $A_k$-word must appear this many times in $w_{h r_i^{(k+1)}}^{(k+1)}$.
 It is not hard to check that this implies $i = 1$ (the maximum number of times an $A_k$-word appears becomes smaller if
$i > 1$, since then the final self-concatenation in the definition of $w_{h r_i^{(k+1)}}^{(k+1)}$ is shorter).
Therefore, $\Phi(w_{\textrm{id}}^{(k+1)}) = w_{h r_1^{(k+1)}}^{(k+1)} = w_h^{(k+1)}$ for some $h \in H_k$.

Recall that $w_{\textrm{id}}^{(k+1)}$ contains every $A_k$-word, and so since $\phi = \alpha_\tau$ and $\pi_{k,h}$ map $w_{\textrm{id}}^{(k+1)}$ to the same word $w_h^{(k+1)}$, it must be the case that $\tau = \pi_{k,h}$. Therefore, $\phi = \phi_{k,h}$, and since $\phi$ was arbitrary, we are done.

\end{proof}

By Lemmas~\ref{pres1} and \ref{pres2}, every $\phi \in \textrm{Aut}(X, \sigma)$ can be written as $\sigma^j \phi_{k,h}$ for some $j \in \mathbb{Z}$, $k \in \mathbb{N}$, and $h \in H_k$, and so $\grp$ is isomorphic to $G$, completing the proof.

\end{proof}

\begin{remark}
We can in fact say a little more about $\aut$ for this construction:
$\aut$ is isomorphic to $\mathbb{Z} \times G$, with the $\mathbb{Z}$ corresponding to $\sigma$. To see this, recall that $G$ is isomorphic to $G_{\phi}$, and then consider the map $\alpha \colon G_{\phi} \times \mathbb{Z} \to \aut$ defined by $\alpha(\phi_{h,k},n) = \phi_{h,k}\sigma^{n}$. It is easy to check that since $\sigma$ is in the center of $\aut$, $\alpha$ is a homomorphism. Lemmas~\ref{pres1} and~\ref{pres2} imply that $\alpha$ is surjective, and it is straightforward to check that $\alpha$ is also injective, and hence an isomorphism.
\end{remark}

\begin{remark}
It is natural to wonder whether the somewhat complex block concatenation subshifts could be replaced by the simpler subclass of Toeplitz subshifts in our constructions. In general this is not possible, since the automorphism group of a Toeplitz subshift is always abelian (see~\cite{DDMP2}).
\end{remark}

\begin{remark}
In Example 3.9 from \cite{BLR}, they construct a minimal subshift $X$ where the additive group of rationals $\mathbb{Q}$ embeds into $\aut$ (and outline alterations which would make this embedding an isomorphism). Since $\faktor{\mathbb{Q}}{\mathbb{Z}}$ is countable and locally finite, Theorem~\ref{mainthm} provides a different minimal subshift $X$ with $\grp = \faktor{\mathbb{Q}}{\mathbb{Z}}$.
\end{remark}


\bibliographystyle{plain}
\bibliography{LocFin}

\begin{thebibliography}{10}

\bibitem{BLR}
Mike Boyle, Douglas Lind, and Daniel Rudolph.
\newblock The automorphism group of a shift of finite type.
\newblock {\em Trans. Amer. Math. Soc.}, 306(1):71--114, 1988.

\bibitem{BD}
Paulina Checchi~Bernales and Sebasti\'{a}n Donoso.
\newblock Strong orbit equivalence and superlinear complexity, submitted.

\bibitem{CKlinear}
Van Cyr and Bryna Kra.
\newblock The automorphism group of a shift of linear growth: beyond
  transitivity.
\newblock {\em Forum Math. Sigma}, 3:Paper No. e5, 27, 2015.

\bibitem{CKstretch}
Van Cyr and Bryna Kra.
\newblock The automorphism group of a minimal shift of stretched exponential
  growth.
\newblock {\em J. Mod. Dyn.}, 10:483--495, 2016.

\bibitem{CKquad}
Van Cyr and Bryna Kra.
\newblock The automorphism group of a shift of subquadratic growth.
\newblock {\em Proc. Amer. Math. Soc.}, 144(2):613--621, 2016.

\bibitem{CKslow}
Van Cyr and Bryna Kra.
\newblock The automorphism group of a shift of slow growth is amenable.
\newblock {\em Ergodic Theory Dynam. Systems}, 40(7):1788--1804, 2020.

\bibitem{DDMP1}
Sebasti\'{a}n Donoso, Fabien Durand, Alejandro Maass, and Samuel Petite.
\newblock On automorphism groups of low complexity subshifts.
\newblock {\em Ergodic Theory Dynam. Systems}, 36(1):64--95, 2016.

\bibitem{DDMP2}
Sebastian Donoso, Fabien Durand, Alejandro Maass, and Samuel Petite.
\newblock On automorphism groups of {T}oeplitz subshifts.
\newblock {\em Discrete Anal.}, pages Paper No. 11, 19, 2017.

\bibitem{grigorp}
R.~I. Grigorchuk.
\newblock Degrees of growth of {$p$}-groups and torsion-free groups.
\newblock {\em Mat. Sb. (N.S.)}, 126(168)(2):194--214, 286, 1985.

\bibitem{grigorgap}
Rostislav Grigorchuk.
\newblock On the gap conjecture concerning group growth.
\newblock {\em Bull. Math. Sci.}, 4(1):113--128, 2014.

\bibitem{gromov}
Mikhael Gromov.
\newblock Groups of polynomial growth and expanding maps.
\newblock {\em Inst. Hautes \'{E}tudes Sci. Publ. Math.}, (53):53--73, 1981.

\bibitem{HK}
Frank Hahn and Yitzhak Katznelson.
\newblock On the entropy of uniquely ergodic transformations.
\newblock {\em Trans. Amer. Math. Soc.}, 126:335--360, 1967.

\bibitem{LohGGT}
Clara L\"{o}h.
\newblock {\em Geometric group theory}.
\newblock Universitext. Springer, Cham, 2017.
\newblock An introduction.

\bibitem{Pset}
K.~E. Petersen.
\newblock A topologically strongly mixing symbolic minimal set.
\newblock {\em Trans. Amer. Math. Soc.}, 148:603--612, 1970.

\bibitem{RobinsonGT}
Derek J.~S. Robinson.
\newblock {\em A course in the theory of groups}, volume~80 of {\em Graduate
  Texts in Mathematics}.
\newblock Springer-Verlag, New York, second edition, 1996.

\bibitem{saloFG}
Ville Salo.
\newblock Toeplitz subshift whose automorphism group is not finitely generated.
\newblock {\em Colloq. Math.}, 146(1):53--76, 2017.

\bibitem{salofull}
Ville Salo.
\newblock A note on subgroups of automorphism groups of full shifts.
\newblock {\em Ergodic Theory Dynam. Systems}, 38(4):1588--1600, 2018.

\bibitem{ST}
Yehuda Shalom and Terence Tao.
\newblock A finitary version of {G}romov's polynomial growth theorem.
\newblock {\em Geom. Funct. Anal.}, 20(6):1502--1547, 2010.

\bibitem{wilson1}
John~S. Wilson.
\newblock On the growth of residually soluble groups.
\newblock {\em J. London Math. Soc. (2)}, 71(1):121--132, 2005.

\bibitem{wilson2}
John~S. Wilson.
\newblock The gap in the growth of residually soluble groups.
\newblock {\em Bull. Lond. Math. Soc.}, 43(3):576--582, 2011.

\end{thebibliography}

\end{document}